\documentclass{amsart}
\usepackage{amssymb}
\usepackage{amsfonts}
\usepackage[all,tips]{xy}
\newtheorem{theorem}{Theorem}
\theoremstyle{plain}

\newtheorem{definition}{Definition}
\newtheorem{example}{Example}

\newtheorem{proposition}{Proposition}

\numberwithin{equation}{section}

\begin{document}
\title[Coverings and Actions of Structured Lie Groupoids I.]{
Coverings and Actions of Structured Lie Groupoids I.}
\author{M. Habil Gursoy, Ilhan Icen and A. Fatih Ozcan}
\address{Department of Mathematics \\
Science and Art Faculty \\
Inonu University \\
Malatya / TURKEY} \email{mhgursoy@gmail.com (Corresponding author) }
\email{iicen@inonu.edu.tr }
\email{afozcan@inonu.edu.tr }
\date{XX XX, XXXX}

\begin{abstract}
In this work we deal with coverings and actions of Lie
group-groupoids being a sort of the structured Lie groupoids.
Firstly, we define an action of a Lie group-groupoid on some Lie
group and the smooth coverings of Lie group-groupoids. Later, we
show the equivalence of the category of smooth actions of Lie
group-groupoids on Lie groups  and the category of smooth coverings
of Lie group-groupoids.
\end{abstract}

\maketitle

\section{Introduction}

The theory of covering space has an important role in the algebraic
topology. When studied on categories and groupoids, the concept of
covering is meaningful by investigation of relationships between
fundamental groupoids of covering spaces and those of the base
spaces.

The first papers in this area were written by Brown and Higgins
\cite{higgins,Brw1, Brw2}. Brown defined the fundamental groupoid
$\pi_{1}X$ associated to a topological space $X$ and obtained the
covering morphism $\pi_{1}p:\pi_{1}\widetilde{X}\rightarrow
\pi_{1}X$ of groupoids for a given covering map
$p:\widetilde{X}\rightarrow X$ of topological spaces. Thus he proved
the equivalence of the category $TCov(X)$ of coverings of $X$ and
the category $GdCov(\pi_{1}X)$ of coverings of fundamental groupoid
$\pi_{1}X$, where $X$ has universal covering space \cite{Brw2}.

Further, the relation between notions of covering and action for
groupoids was studied by Gabriel and Zisman \cite{gabriel}. They
proved the equivalence of the category $GdCov(G)$ of coverings of a
groupoid $G$ and the category $GdOp(G)$ of actions of $G$ on the
sets.

Another concept considered in this paper is group-groupoid notion
which is a group object in the category of groupoids. It was defined
by Brown and Spencer \cite{Brw6}. They proved that if $X$ is a
topological group, then the fundamental groupoid $\pi_{1}X$ becomes
a group-groupoid.

After defining the topological groupoid by Ehresmann all of the
above results were also given for topological groupoids \cite{Brw1}.
In \cite{IFH}, it was proved that the topological group structure of
a topological group-groupoid how was lifted to a topological
universal covering groupoid.

In this work we deal with smoothness of all these results. It is
well-known that for a connected smooth manifold $M$ there exist a
universal covering manifold $\widetilde{M}$ and a smooth covering
map $p:\widetilde{M}\rightarrow M$. Also, from the manifold theory
there exist a simply connected universal covering Lie group
$\widetilde{G}$ of a connected Lie group $G$ and a smooth covering
map $p:\widetilde{G}\rightarrow G$ which is a Lie group homomorphism
at the same time.

By using these facts, we show that the fundamental Lie groupoid
$\pi_{1}M$ associated to a connected Lie group $M$ is a Lie
group-groupoid. Thus we constitute the category $LGCov(M)$ of smooth
coverings of a Lie group $M$ and the category $LGGdCov(\pi_{1}M)$ of
coverings of Lie group-groupoid $\pi_{1}M$. Then we prove the
equivalence of these categories. Also, we show that the category
$LGGdCov(G)$ of coverings of a Lie group-groupoid $G$ and the
category $LGGdOp(G)$ of actions of Lie group-groupoid $G$ on Lie
groups.

Throughout the paper, all our manifolds we consider are assumed to
be smooth and second countable.

\section{Lie Group-Groupoids}

In this section, we will give the basic definitions and concepts
related to Lie group-groupoids. But we recall, in advance, the basic
facts about groupoids and Lie groupoids.

A groupoid is a category in which every arrow is invertible. More
precisely, a groupoid consists of two sets $G$ and $G_{0}$ called
the set of arrows (or morphisms) and the set of objects of groupoid
respectively, together with two maps $\alpha, \beta :G\rightarrow
G_{0}$ called source and target maps respectively, a map
$\epsilon:G_{0}\rightarrow G$,$x\mapsto \epsilon(x)=1_{x}$ called
the object map, an inverse map $i:G\rightarrow G$,$a\mapsto a^{-1}$
and a composition $G_{2}=G ~_{\alpha}\times_{\beta}G\rightarrow
G$,$(b,a)\mapsto b\circ a$ defined on the pullback
\begin{center}
$G~_{\alpha}\times_{_{\beta}}G=\{(b,a)\mid \alpha(b)=\beta(a)\}$.
\end{center}

These maps should satisfy the following conditions:

\begin{enumerate}
\item $\alpha (b\circ a)=\alpha (a)$ and $\beta (b\circ a)=\beta
(b)$, for all $(b,a)\in G_{2}$,

\item $c\circ (b\circ a)=(c\circ b)\circ a$ such that $\alpha
(b)=\beta(a)$ and $\alpha (c)=\beta(b)$, for all $a,b,c\in G$,

\item $\alpha(1_{x})=\beta (1_{x})=x$, for all $x\in G_{0}$,

\item $a\circ 1_{\alpha(a)}=a$ and $1_{\beta(a)}\circ a=a$, for all
$a\in G$,

\item $\alpha (a^{-1})=\beta (a)$ and $\beta (a^{-1})=\alpha (a)$,
$a^{-1}\circ a=1_{\alpha(a)}$ and $a\circ a^{-1}=1_{\beta (a)}$
\cite{Br-Icen}.
\end{enumerate}

Let $G$ be a groupoid. For all $x,y\in G_{0}$, we denote $G(x,y)$
the set of all arrows $a\in G$ such that $\alpha (a)=x$ and $\beta
(a)=y$. For $x\in G_{0}$, we write $St_{G}x$ for the set of all
arrows started at $x$, and $CoSt_{G}x$ for the set of all arrows
ended at $x$. The object or vertex group at $x$ is $G\{x\}=\{a\in
G\mid \alpha(a)=\beta(a)=x\}$.

Let $G$ and $H$ be two groupoids. A groupoid morphism from $H$ to
$G$ is a pair $(f,f_{0})$ of maps $f:H\rightarrow G$ and
$f_{0}:H_{0}\rightarrow G_{0}$ such that $\alpha_{G}\circ
f=f_{0}\circ \alpha_{H}$, $\beta_{G}\circ f=f_{0}\circ \beta_{H}$
and $f(b\circ a)=f(b)\circ f(a)$ for all $(b,a)\in H_{2}$
\cite{Br-Icen,Brw2}.

Throughout the work we shall assume that the set of objects $G_{0}$
and $\alpha$-fibers $\alpha^{-1}(x)=St_{G}x$, $x\in G_{0}$ are
Hausdorff.

\begin{definition}
A groupoid $G$ over $G_{0}$ is called Lie groupoid if $G$ and
$G_{0}$ are manifolds, $\alpha$ and $\beta$ are surjective
submersions and the composition map is smooth \cite{Br-Icen}.
\end{definition}

It follows that, $\epsilon$ is an immersion, the inverse map is a
diffeomorphism, the sets $St_{G}x$, $CoSt_{G}x$ and $G(x,y)$ are
closed submanifolds of $G$ for all $x,y\in G_{0}$ and all vertex
groups are Lie groups. Also since $\alpha$ and $\beta$ are
submersions, $G_{2}$ is a closed submanifold of $G\times G$
\cite{McKenzie}.

The left-translation (right translation) for $a\in G(x,y)$ is the
map $L_{a}:CoSt_{G}x \rightarrow CoSt_{G}y$, $b\mapsto a\circ b$
($R_{a}:St_{G}y \rightarrow St_{G}x$, $b\mapsto b\circ a$) which is
a  diffeomorphism \cite{Br-Icen}.

\begin{example}
Let $M$ be a manifold. The product manifold $M\times M$ is a Lie
groupoid over $M$ in the following way: $\alpha$ is the second
projection and $\beta$ is the first projection; $1_{x}=(x,x)$ for
all $x\in M$ and $(x,y)\circ (y,z)=(x,z)$ \cite{Moer-Mrcun}.
\end{example}

\begin{definition}
A morphism between Lie groupoids $H$ and $G$ is a groupoid morphism
$(f,f_{0})$ such that $f$ and $f_{0} $ are smooth \cite{McKenzie}.
\end{definition}

Now, we can introduce definition of Lie group-groupoids.

\begin{definition}
A Lie group-groupoid is a Lie groupoid endowed with a structure of
Lie group such that the addition $m:G\times G\rightarrow G$,
$(a,b)\mapsto a+b$, the unit map $e:\ast \rightarrow G$, $\ast
\mapsto e(\ast )=1_{e}$ and inverse map $\overline{u}:G\rightarrow
G$, $a\mapsto -a$, which are the structure maps of Lie group, are
Lie groupoid morphisms. Also there exists an interchange law
\begin{equation*}
(b\circ a)+(d\circ c)=(b+d)\circ (a+c).
\end{equation*}
\end{definition}

\begin{definition}
Let $G$ and $H$ be two Lie group-groupoids. A morphism
$f:H\rightarrow G$ of Lie group-groupoids is a morphism of
underlying Lie groupoids preserving the Lie group structure, i.e.,
$f(a+b)=f(a)+f(b)$ for $a,b\in H$.
\end{definition}

\begin{example}\label{ornek.3.2.1}
Let $G$ be a Lie group. Then we constitute a Lie group-groupoid
$G\times G$ with object manifold $G$ as the following way:

A morphism from an object $x$ to another one $y$ is a pair of
$(y,x)$. The source map is defined by $\alpha (y,x)=x$, the target
map is defined by $\beta (y,x)=y$, the object map is defined by
$x\mapsto (x,x)$ for any $x\in G $, the inverse of $(y,x)$ is
defined by $(x,y)$ and the composition is defined by $(z,y)\circ
(y,x)=(z,x)$ for $(y,x)$, $(z,y)\in G\times G$. Since $G$ is a Lie
group, $G\times G$ is also a Lie group with the operation
$(x,y)+(z,t)=(x+z,y+t)$ defined by the operation of $G$. The unit
element of this group is $(e,e)$ where $e$ is the unit element of
$G$, and the inverse in the group of $(y,x)$ is $(-y,-x)$. Also
$G\times G$ is the Lie groupoid called the banal groupoid. Now let
us show that the group structure maps of $G\times G$ are groupoid
morphisms.

For $m:(G\times G)\times (G\times G)\rightarrow G\times G,$%
\begin{align*}
m(((z,y),(z^{^{\prime}},y^{^{\prime}}))\circ((y,x),(y^{^{\prime}},x^{^{\prime
}})))  &  =m((z,y)\circ(y,x),(z^{^{\prime}},y^{^{\prime}})\circ(y^{^{\prime}%
},x^{^{\prime}}))\\
&  =m((z,x),(z^{^{\prime}},x^{^{\prime}}))\\
&  =(z+z^{^{\prime}},x+x^{^{\prime}})
\end{align*}
and
\begin{align*}
m((z,y),(z^{^{\prime}},y^{^{\prime}}))\circ m((y,x),(y^{^{\prime}}%
,x^{^{\prime}}))  &  =((z,y)+(z^{^{\prime}},y^{^{\prime}}))\circ
((y,x)+(y^{^{\prime}},x^{^{\prime}}))\\
&  =(z+z^{^{\prime}},y+y^{^{\prime}})\circ(y+y^{^{\prime}},x+x^{^{\prime}})\\
&  =(z+z^{^{\prime}},x+x^{^{\prime}})
\end{align*}

Similarly, it can be shown that the unit map and the inverse map of
the group are also groupoid morphisms. Furthermore, since the group
structure maps of $G\times G $ are defined by the operations of Lie
group $G$, they are also smooth. Consequently, $G\times G$ is a Lie
group-groupoid.
\end{example}

\bigskip

This example defines a functor from the category $LGrp$ of Lie
groups to the category $LGGd$ of Lie group-groupoids. Let us give it
by the following proposition.

\begin{proposition}
There exists a functor $\Omega :LGrp\rightarrow LGGd$ from the
category $LGrp $ of Lie groups to the category $LGGd$ of Lie
group-groupoids.
\end{proposition}

\begin{proof}
Let $G$ be a Lie group. Then, from Example \ref{ornek.3.2.1},
$G\times G$ is a Lie group-groupoid. If $f:G\rightarrow H$ is a
morphism of Lie groups, then $\Omega (f):G\times G\rightarrow
H\times H$ is a morphism of Lie group-groupoids. Indeed, $\Omega
(f)$ is defined by $(y,x)\mapsto (f(y),f(x)) $ and $\Omega (f)$
preserves the group structure. That is,
\begin{eqnarray*}
\Omega (f)((y,x)+(y^{^{\prime }},x^{^{\prime }})) &=&\Omega
(f)(y+y^{^{\prime }},x+x^{^{\prime }}) \\
&=&(f(y+y^{^{\prime }}),f(x+x^{^{\prime }})) \\
&=&(f(y)+f(y^{^{\prime }}),f(x)+f(x^{^{\prime }})) \\
&=&(f(y),f(x))+(f(y^{^{\prime }}),f(x^{^{\prime }})) \\
&=&\Omega (f)(y,x)+\Omega (f)(y^{^{\prime }},x^{^{\prime }}).
\end{eqnarray*}
By the same idea, $\Omega (f)((z,y)\circ (y,x))=$ $\Omega
(f)(z,x)=(f(z),f(x))$ and $\Omega (f)(z,y)\circ \Omega
(f)(y,x)=(f(z),f(y))\circ (f(y),f(x))=(f(z),f(x))$. Hence we obtain
that $\Omega (f)((z,y)\circ (y,x))=$ $\Omega (f)(z,y)\circ \Omega
(f)(y,x)$. Thus $\Omega (f)$ preserves the groupoid structure.
$\Omega (f)=(f,f)$ is smooth, since $f$ is smooth. Consequently
$\Omega (f)$ is a morphism of Lie group-groupoids.
\end{proof}

\begin{theorem}
The transitive component $C_{e}(G)$ of unit element $e$ in a Lie
group-groupoid $G$ is a Lie subgroup-groupoid which has the
structure of normal Lie subgroup with the addition of Lie group.
\end{theorem}

\begin{proof}
We know that $C_{e}(G)$ is the subgroupoid of $G$. For all object
selected in $C_{e}(G)_{0}$, let us consider the arrow $T_{x}\in
G(x,e)$. Let $a,b\in C_{e}(G)$ be arrows, where $a\in G(x,y)$ and
$b\in G(x^{^{\prime }},y^{^{\prime }})$. Thus, we obtain $a-b\in
G(x-x^{^{\prime }},y-y^{^{\prime }})$, where $T_{x}-T_{x^{^{\prime
}}}\in G(x-x^{^{\prime }},e)$ and $T_{y}-T_{y^{^{\prime }}}\in
G(y-y^{^{\prime }},e)$. Hence it follows that $x-x^{^{\prime
}},y-y^{^{\prime }}\in C_{e}(G)_{0}$ and $a-b\in C_{e}(G)$. In other
words, $C_{e}(G)$ is a subgroup. There are structures of submanifold
on $C_{e}(G)_{0}$ and $C_{e}(G)$. Also, the addition in $C_{e}(G)$
is the addition of $G$, so it is smooth. Thus, $C_{e}(G)$ is a Lie
group. Furthermore, the structure maps of the subgroupoid $C_{e}(G)$
are restrictions of structure maps in $G$, so they are smooth too.
Consequently $C_{e}(G)$ is a Lie group-groupoid. Now let us show
that $C_{e}(G)$ is normal. For any $a\in G(x,y)$, let $a\in
C_{e}(G)$ and $g\in G(w,z)$. Thus, for any $g\in G(w,z)$, $T_{x}\in
G(x,e)$ and $-g\in G(-w,-z)$ we have $g+T_{x}-g\in G(w+x-w,z+e-z)$,
and hence $g+T_{x}-g\in G(w+x-w,e)$. It follows $g+a-g\in
G(w+x-w,z+y-z)$. That is, we obtain $g+a-g\in C_{e}(G)$. We deduce
that $C_{e}(G)$ is a normal Lie subgroup. Consequently, $C_{e}(G)$
is a Lie subgroup-groupoid.
\end{proof}

\begin{theorem}
All characteristic groups in a Lie group-groupoid $G$ are linearly
diffeomorphic to each other.
\end{theorem}

\begin{proof}
It is enough to show that for any $x\in G_{0}$ the object group
$G\{x\}$ is diffeomorphic to the vertex group $G\{e\}$. From the
definition of left translation, we can write
$L_{1_{x}}:G\{e\}\rightarrow G\{x\}$, $a\mapsto 1_{x}+a$. On the
other hand,
\begin{equation*}
L_{1_{x}}(b\circ a)=1_{x}+(b\circ a)=(1_{x}\circ 1_{x})+(b\circ
a)=(1_{x}+b)\circ (1_{x}+a)
\end{equation*}
by the interchange law. Therefore, $L_{1_{x}}$ is a homomorphism of
groups. Since the operations $+$ and $\circ $ are the operations of
Lie group and Lie groupoid, respectively, they are smooth. That is,
$L_{1_{x}}$ is smooth. Furthermore, there exists inverse of
$L_{1_{x}}$, and it is also smooth.
\end{proof}

A Lie group-groupoid $G$ is called transitive, $1$-transitive or
simply transitive, if the underlying Lie groupoid of $G$ is
transitive, $1$ -transitive or simply transitive, respectively.

\section{\protect\bigskip Coverings and Actions of Lie Group-Groupoids}

It is useful to present the definition of covering morphism of Lie
groupoids before the definition of covering morphism of Lie
group-groupoids.

\begin{definition}
Let $p:\widetilde{G}\rightarrow G$ be a morphism of Lie groupoids.
For each $\widetilde{x}\in \widetilde{G}_{0}$, if the restriction
$\widetilde{G}_{\widetilde{x}}\rightarrow G_{p(\widetilde{x})}$ of
$p$ is a diffeomorphism, $p$ is called the covering morphism of Lie
groupoids. Then $\widetilde{G}$ is called the covering of Lie
groupoid $G$.
\end{definition}

Let us give an equivalent criterion to the covering of Lie
groupoids.

Let $p:H \rightarrow G$ be a covering morphism of Lie groupoids.
Take the pullback
\begin{center}
$G~_{\alpha}\times_{p_{0}}H_{0}=\{(a,x)\in G\times H_{0}\mid
\alpha(a)=p_{0}(x)\}.$
\end{center}
Since $\alpha$ is a submersion, $G~_{\alpha}\times_{p_{0}}H_{0}$ is
a manifold. Then the map
$s_{p}:G~_{\alpha}\times_{p_{0}}H_{0}\rightarrow H$ is the lifting
function assigning the unique element $h\in H_{x}$ to the pair
$(a,x)$ such that $p(h)=a$. It is clear that $s_{p}$ is inverse of
the map $(p,\alpha):H\rightarrow G~_{\alpha}\times_{p_{0}}H_{0}$.

Thus, the morphism $p:H\rightarrow G$ is covering morphism of Lie
groupoids iff the morphism $(p,\alpha)$ is a diffeomorphism.

\begin{definition}
A morphism $ f:H\rightarrow G$ of Lie group-groupoids is called a
covering morphism of Lie group-groupoids if it is a covering
morphism of underlying Lie groupoids.
\end{definition}

\begin{proposition}\label{Onerme3.2.2}
Let $M$ be a connected Lie group. Then $\pi _{1}M$ is a Lie
group-groupoid.
\end{proposition}

\begin{proof}
From \cite{Brw1}, it is known that $\pi _{1}M$ is a group-groupoid.
Let us denote the atlas making the manifold $M$ smooth and
consisting of the liftable charts by $\mathcal{A}$. Since $M$ is
connected manifold, the fundamental groupoid $\pi _{1}M$ is a Lie
groupoid \cite{Danesh}. So there exists an atlas $\mathcal{A}$
lifted from $\mathcal{A}$ over $\pi _{1}M$. Further this atlas makes
it smooth manifold. Now we can show that $\pi _{1}M$ is a Lie
group-groupoid. For this, it is enough to show that the group
operation $\pi _{1}m:\pi _{1}M\times \pi _{1}M\rightarrow \pi _{1}M$
is smooth. Since $m:M\times M\rightarrow M$ is the addition of Lie
group, it is smooth. In fact, it is easily seen by the following
diagram.
\[
\xymatrix{M\times M \ar[rr]^{m} \ar[dd]_{\varphi \times \varphi }&&M\ar[dd]^{\varphi}\\\\
{\mathbb{R}^{2n}}\ar[rr]^{pr_{1}}&& \mathbb{R}^{n}}
\]
At this diagram, $\varphi $ is a coordinat chart with domain $U$
selected from the liftable atlas of $M$. So $\varphi $ is a
diffeomorphism onto the open subset $\varphi (U)\subset
\mathbb{R}^{n}$. From here, $\varphi \times \varphi $ is also
smooth. Furthermore, $pr_{1}$ is smooth, because it is a projection
onto the first factor. Thus, $m=\varphi ^{-1}\circ pr_{1}\circ
(\varphi \times \varphi )$. That is, $m$ is smooth map. Since $\pi
_{1}M$ is a Lie groupoid over $M$, and it is the covering manifold
of the product manifold $M\times M$, then we can give the following
diagram:
\[
\xymatrix{{\pi_{1}(M)\times \pi_{1}(M)} \ar[rr]^{\pi_{1}m} \ar[dd]_{\widetilde{\varphi} \times \widetilde{\varphi} }&&{\pi_{1}M}\ar[dd]^{\widetilde{\varphi}}\\\\
{\mathbb{R}^{4n}}\ar[rr]^{pr_{1,2}}&& \mathbb{R}^{2n}}
\]
,which is lifting of the above diagram. $\widetilde{\varphi }$ is
the coordinat chart lifted from the coordinat chart $\varphi $. So
$\widetilde{\varphi }$ is smooth. This brings the map
$\widetilde{\varphi }\times \widetilde{\varphi }$ is smooth too.
$pr_{1,2}$ is also smooth, because it is a projection. Hence $\pi
_{1}m$ is smooth. Therefore, $\pi _{1}M$ is a Lie group-groupoid.
\end{proof}

From the following proposition, Lie group-groupoid $\pi _{1}G$ is
functorial.

\begin{proposition}\label{onerme3.2.4}
Let $f:H\rightarrow G$ be a morphism of connected Lie groups. A
morphism $\pi _{1}f:\pi _{1}H\rightarrow \pi _{1}G$ induced from $f$
is a morphism of Lie group-groupoids.
\end{proposition}

\begin{proof}
It was proved that $\pi _{1}f$ is the morphism of group-groupoids in
\cite{Mck1}. For this reason, it is enough to show that $\pi _{1}f$
is smooth. Since $f:H\rightarrow G$ is smooth, we have $f(U)\subset
V$ for the charts $(U,\varphi )$ and $(V,\psi )$\ on $H$ and $G$,
respectively. And hence $\psi \circ f\circ \varphi ^{-1}:\varphi
(U)\rightarrow \psi (V)$ is a smooth map.
\[
\xymatrix{H \ar[rr]^{f} \ar[dd]_{\varphi }&&G\ar[dd]^{\psi}\\\\
{\mathbb{R}^{n}}\ar[rr]^{I}&& \mathbb{R}^{n}}
\]
There exist the fundamental Lie groupoids $\pi _{1}H$ and $\pi
_{1}G$ corresponding to $H$ and $G$, respectively. So we can define
groupoid morphism $\pi _{1}f$ as $\widetilde{\psi }^{-1}\circ
Id\circ \widetilde{\varphi }$ by coordinat charts
$(\widetilde{U},\widetilde{\varphi })$ and
$(\widetilde{V},\widetilde{\psi })$ which are liftings of the
coordinat charts $(U,\varphi )$ and$(V,\psi )$, respectively.
$\widetilde{\psi }$ and $\widetilde{\varphi }$ are smooth, because
they are the chart maps of Lie groupoids $\pi _{1}G$ and $\pi
_{1}H$, respectively. Since the $Id$ is the unit map, it is smooth.
Thus $\pi _{1}f$ is smooth.
\end{proof}

\begin{proposition}
Let $p:\widetilde{M}\rightarrow M$ be a covering morphism of
connected Lie groups. Then the morphism $\pi _{1}p:\pi
_{1}\widetilde{M}\rightarrow \pi _{1}M$ is a covering morphism of
Lie group-groupoids.
\end{proposition}

\begin{proof}
Let $p:\widetilde{M}\rightarrow M$ be covering morphism of the
connected Lie groups. It is obvious that $p$ is smooth and a Lie
group homomorphism. Since $\widetilde{M}$ and $M$ are connected
smooth manifolds, from Proposition \ref{Onerme3.2.2}, $\pi
_{1}\widetilde{M}$ and $\pi _{1}M$ are Lie group-groupoids. It is
known that $\pi _{1}p:\pi _{1}\widetilde{M}\rightarrow \pi _{1}M$ is
covering morphism of groupoids \cite{Brw2}. Also by Proposition
\ref{onerme3.2.4}, $\pi _{1}p$ is the morphism of Lie
group-groupoids. Now let us show that $\pi _{1}p$ is covering
morphism of Lie group-groupoids. $\pi _{1}p $ is smooth, because it
is the morphism of Lie group-groupoids. Since $\alpha $ is the
source map of Lie group-groupoid, it is obvious that it is smooth
too. So the map $(\pi _{1}p,\alpha ):\pi
_{1}\widetilde{M}\rightarrow \pi _{1}M~_{\alpha }\times
_{p}\widetilde{M\text{ }}$is smooth. Further, it is bijection,
because $\pi _{1}p$ is the covering morphism of group-groupoids.
Hence there exists an inverse $s_{\pi _{1}p}:\pi _{1}M~_{\alpha
}\times _{p}\widetilde{M\text{ }}\rightarrow \pi _{1}\widetilde{M}$
of $(\pi _{1}p,\alpha )$. $s_{\pi _{1}p}$ is the function assigning
to the unique homotopy class $[h]_{\widetilde{x}}$ started at
$\widetilde{x}$ of smooth paths $h$ each pair $([a],\widetilde{x})$
such that $\pi _{1}p([h])=[a]$. By the homotopy lifting property and
the unique lifting property, it is obvious that $s_{\pi _{1}p}$ is
well-defined. Also, we can write $s_{\pi _{1}p}$ as the composition
of the smooth maps similar to the following diagram
\[
\xymatrix{\pi_{1}M~_{\alpha}\times
_{p}\tilde{M} \ar[r]^{I\times \epsilon} &{\pi_{1}M\times \pi_{1}\widetilde{M}} \ar[r]^{I\times L_{\widetilde{a}}} &{\pi_{1}M\times \pi_{1}\widetilde{M}} \ar[rr]^{pr_{2}} &&\pi_{1}\widetilde{M}\\
{([a],\widetilde{x})} \ar@{|->}[r]^{} & {([a],[1_{\widetilde{x}}])}
\ar@{|->}[r]^{} &{([a],[\widetilde{a}])} \ar@{|->}[r]^{}
&{[\widetilde{a}]}}
\]
Hence $s_{\pi _{1}p}$ is smooth. Thus, $(\pi _{1}p,\alpha )$ is a
diffeomorphism. Consequently, $\pi _{1}p$ is the covering morphism
of Lie group-groupoids.
\end{proof}

If $M$ is connected Lie group, by Proposition \ref{Onerme3.2.2} the
fundamental groupoid $\pi _{1}M$ is a Lie group-groupoid. Thus, we
obtain a category $LGGdCov(\pi _{1}M)$. Objects of this category are
covering morphisms $p:\widetilde{G}\rightarrow \pi _{1}M$ of Lie
group-groupoids, and a morphism from an object
$p:\widetilde{G}\rightarrow \pi _{1}M$ to an object
$p:\widetilde{H}\rightarrow \pi _{1}M$ is a morphism
$r:\widetilde{G}\rightarrow \widetilde{H}$ of Lie group-groupoids
such that $p=q\circ r$, where $\widetilde{M}=\widetilde{G}_{0}$ is
connected manifold.

More generally; let $G$ be a Lie group-groupoid. Then we obtain a
category $LGGdCov(G)$ whose objects are covering morphisms
$p:H\rightarrow G$ of Lie group-groupoids. In this category, a
morphism from an object $p:H\rightarrow G$ to an object
$q:K\rightarrow G$ is a morphism $r:H\rightarrow K$ of Lie
group-groupoids satisfying the condition $p=q\circ r$.

Let $\widetilde{M}$ and $M$ be connected Lie groups. Then we obtain
a category $LGCov(M)$. Its objects are covering morphisms
$p:\widetilde{M}\rightarrow M$ of Lie groups. A morphism from an
object $p:\widetilde{M}\rightarrow M$ to an object
$q:\widetilde{N}\rightarrow M$ is a morphism
$r:\widetilde{M}\rightarrow \widetilde{N}$ of Lie groups such that
$p=q\circ r $.

Let us now state a proposition from \cite{MHG} to be necessary for
the proof of the following theorem.

\begin{proposition}\label{exist of smooth manifold}
Let $M$ be a connected manifold and let $q:\widetilde{G}\rightarrow
\pi_{1}M$ be covering morphism of groupoids. Let
$\widetilde{M}=\widetilde{G}_{0}$ and $p=q_{0}:\widetilde{M}
\rightarrow M$. Let $\mathcal{A}$ denotes an atlas consisting of the
liftable charts. Then the smooth structure over $\widetilde{M}$ is
the unique structure such that the followings are hold:
\begin{enumerate}
\item $p:\widetilde{M} \rightarrow M$ is a covering map.

\item There exists an isomorphism
$r:\tilde{G}\rightarrow\pi_{1}\tilde{M}$ which is the identical on
objects such that the following diagram is commutative:
\end{enumerate}
\[
\xymatrix{& {\pi_{1}\widetilde{M}} \ar[d]^{p} \\ \tilde G
\ar[ur]^{r} \ar[r]_{q} & \pi_{1}M}
\]

\end{proposition}

Let us now give first main result of this paper.

\begin{theorem}
Let $M$ be a connected Lie group. Then the category $LGCov(M)$ of
the smooth coverings of Lie group $M$ is equivalent to the category
$LGGdCov(\pi_{1}M)$ of the coverings of Lie group-groupoid
$\pi_{1}M$.
\end{theorem}
\begin{proof}
Let us define a functor $\Gamma :LGCov(M)\rightarrow LGGdCov(\pi
_{1}M)$ as follows. Let $M$, $\widetilde{M}$ be connected Lie
groups. By Proposition \ref{onerme3.2.4}, if
$p:\widetilde{M}\rightarrow M$ is a covering morphism of Lie groups,
then $\pi _{1}p:\pi _{1}\widetilde{M}\rightarrow \pi _{1}M$ is a
covering morphism of Lie group-groupoids. Hence $\Gamma (p)=\pi
_{1}p$ is a covering morphism of Lie group-groupoids. If
$r:\widetilde{M}\rightarrow \widetilde{N}$ is a morphism of covering
morphisms of Lie groups from $p:\widetilde{M}\rightarrow M$ to
$q:\widetilde{N}\rightarrow M$, by the definition of category
$LGCov(M)$, $r$ is also covering morphism of Lie groups, and since
$\widetilde{M}$, $\widetilde{N}$ are connected, $\pi _{1}r$ is also
covering morphism of Lie group-groupoids. Obviously $\Gamma (r)$ is
a morphism of covering morphisms of Lie group-groupoids from $\pi
_{1}p$ to $\pi _{1}q$. Let $r^{^{\prime }}:\widetilde{N}\rightarrow
\widetilde{P}$ be another morphism of the covering morphisms of Lie
groups, where $q^{^{\prime }}:\widetilde{P}\rightarrow M$. Since $r$
and $r^{^{\prime }}$ are covering morphisms of Lie groups, the
composition $r^{^{\prime }}\circ r:\widetilde{M}\rightarrow
\widetilde{P}$ is also covering morphism of Lie groups and is
clearly morphism of covering morphisms of Lie groups from $p$ to
$q^{^{\prime }}$. Also, $\pi _{1}(r^{^{\prime }}\circ r)=\pi
_{1}r^{^{\prime }}\circ \pi _{1}r$ is the covering morphism of Lie
group-groupoids, because $\widetilde{M}$ and $\widetilde{P}$ are
connected. Furthermore, it is a morphism of covering morphisms of
Lie group-groupoids from $\pi _{1}p$ to $\pi _{1}q^{^{\prime }}$.
Thus, we have $\Gamma (r^{^{\prime }}\circ r)=\Gamma r^{^{\prime
}}\circ \Gamma r$, so that $\Gamma $ is a functor.

Now let us set a functor $\Phi :LGGdCov(\pi _{1}M)\rightarrow
LGCov(M)$. Suppose $\widetilde{G}_{0}=\widetilde{M}$ and let
$q:\widetilde{G}\rightarrow \pi _{1}M$ be smooth covering morphism
of Lie -group-groupoids. Since $M$ is connected, from Proposition
\ref{exist of smooth manifold}, for
$p=q_{0}:\widetilde{M}\rightarrow M$ there exist a lifted manifold
on $\widetilde{M}$ making $p$ smooth covering map on the underlying
manifolds of $M$ and $\widetilde{M}$, and an isomorphism
$r:\widetilde{G}\rightarrow \pi _{1}\widetilde{M}$. Furthermore,
since $p=q_{0}$ and $q$ are Lie group-groupoid morphisms, $p$ is
also morphism of Lie groups. Thus, $\Phi (q)=q_{0}=p$ is a covering
morphism of Lie groups. Let $f:\widetilde{G}\rightarrow
\widetilde{H}$ be a morphism of smooth covering morphisms of Lie
group-groupoids from $q:\widetilde{G}\rightarrow \pi _{1}M$ to
$q^{^{\prime }}:\widetilde{H}\rightarrow \pi _{1}M$. From
Proposition \ref{Onerme3.2.2}, for the liftable atlas $\mathcal{A}$
on $M$ there exist lifted atlases $\mathcal{\widetilde{A}}_{q}$ and
$\mathcal{\widetilde{A}}_{q^{^{\prime }}}$ on
$\widetilde{M}=\widetilde{G}_{0}$ and
$\widetilde{N}=\widetilde{H}_{0}$, respectively. These atlases
consist of the lifted charts making the manifolds $\widetilde{M}$
and $\widetilde{N}$ are smooth. Let $\widetilde{x}\in \widetilde{M}$
and let $\widetilde{U}_{q^{^{\prime }}}$ be an element of
$\mathcal{\widetilde{A}}_{q^{^{\prime }}}$ including
$f(\widetilde{x})$. Then $U=q^{^{\prime
}}(\widetilde{U}_{q^{^{\prime }}})\in \mathcal{A}$ and $U$ is lifted
to a unique $\widetilde{U}_{q}\in \mathcal{\widetilde{A}}_{q}$,
which contains $\widetilde{x}$. Also
$f(\widetilde{U}_{q})=\widetilde{U}_{q^{^{\prime }}}$. Hence
$f:\widetilde{M}\rightarrow \widetilde{N}$ is smooth. That is, $\Phi
(f)$ is a morphism of covering morphisms of Lie groups. Let
$f^{^{\prime }}:\widetilde{H}\rightarrow \widetilde{H^{^{\prime }}}$
be another morphism from $q^{^{\prime }}:\widetilde{H}\rightarrow
\pi _{1}M$ to $q^{^{\prime \prime }}:\widetilde{H^{^{\prime
}}}\rightarrow \pi _{1}M$. Since $f$ and $f^{^{\prime }}$ are
covering morphisms of Lie group-groupoids, the composition
$f^{^{\prime }}\circ f$ is also covering morphism of Lie
group-groupoids, and clearly it is a morphism of covering morphisms
of Lie group-groupoids from $q$ to $q^{^{\prime \prime }}$.
Furthermore, $\Phi (f^{^{\prime }}\circ f)$ is a morphism of
covering morphisms of Lie groups as above. Thus, we have $\Phi
(f^{^{\prime }}\circ f)=\Phi (f^{^{\prime }})\circ \Phi (f)$, so
that $\Phi $ is a functor.

Now let us show natural equivalences $\Gamma \Phi \simeq
1_{LGGdCov(\pi _{1}M)}$ and $\Phi \Gamma \simeq 1_{LGCov(M)}$. \ Let
$q:\widetilde{G}\rightarrow \pi _{1}M$ and $q^{^{\prime
}}:\widetilde{H}\rightarrow \pi _{1}M$ be covering morphisms of Lie
group-groupoids. Since $M$ is connected, there exist the covering
maps $p=q_{0}:\widetilde{M}\rightarrow M$\ and $p^{^{\prime
}}=q_{0}^{^{\prime }}:\widetilde{N}\rightarrow M$\ of smooth
manifolds and the isomorphisms $r:\widetilde{G}\rightarrow \pi
_{1}\widetilde{M}$ and $r^{^{\prime }}:\widetilde{H}\rightarrow \pi
_{1}\widetilde{N}$. Now let us show that the following diagram is
commutative.
\[
\xymatrix{\tilde{G} \ar[rr]^{r} \ar[dd]_{f}&&\pi_{1}
\tilde{M}\ar[dd]^{\pi_{1} f_{0}}\\\\
{\tilde{H}}\ar[rr]^{r'}&& \pi_{1} \tilde{N}}
\]
Indeed; let $\widetilde{a}$ be an element of $\widetilde{G}$ started
at $\widetilde{x}$ and let $a:I\rightarrow M$ be a representation of
$q(\widetilde{a})\in \pi _{1}M$. Then $a$ induces a morphism $\pi
_{1}a:\pi _{1}I\rightarrow \pi _{1}M$, $\pi _{1}a(\imath
)=q(\widetilde{a})$. Further, $\pi _{1}a$ is lifted to unique
morphism $a^{^{\prime }}:(\pi _{1}I,0)\rightarrow
(\widetilde{G},\widetilde{x})$. Then $r(\widetilde{a})$ is the
equivalent class of the path $a_{0}^{^{\prime }}:I\rightarrow
\widetilde{M}$. Let $\widetilde{b}=f(\widetilde{a})$. We use the
same method in order to obtain $b^{^{\prime }}:(\pi
_{1}I,0)\rightarrow (\widetilde{H},f(\widetilde{x}))$, where
$b=f_{0}(a)$. Since $b^{^{\prime }}$ is determined as unique by $b$,
we obtain $r^{^{\prime }}f(\widetilde{a})=(\pi
_{1}f_{0})r(\widetilde{a})$. This means that we have $\ \Gamma \Phi
\simeq 1_{LGGdCov(\pi _{1}M)}$. Finally, we must show $\Phi \Gamma
\simeq 1_{LGCov(M)}$. But since $\widetilde{M}=(\pi
_{1}\widetilde{M})_{0}$ and the structure manifold of
$\widetilde{M}$ is the lifted manifold, it is obvious that $\Phi
\Gamma =1_{LGCov(M)}$. Thus, the proof is completed.
\end{proof}

Now we will introduce a smooth action of a Lie group-groupoid on a
connected Lie group. First let us give the definition.

\begin{definition}
Let $G$ be a Lie group-groupoid and let $M$ be a Lie group. An
action of Lie group-groupoid $G$\ on Lie group $M$ via Lie group
homomorphism $w:M\rightarrow G_{0}$ consists of the action of the
underlying Lie groupoid of $G$ on the underlying smooth manifold of
$M$ via the submersion $w:M\rightarrow G_{0}$ satisfying the
conditions $w(^{a}x)=\beta (a)$, $^{b}(^{a}x)=~^{(b\circ a)}x$ and
$^{1_{w(x)}}x=x$ such that interchange law
$(^{b}y)+(^{a}x)=~^{b+a}(y+x)$ is hold. Such an action is denoted by
$(M,w)$.
\end{definition}

\begin{example}
If $G$ is a Lie group-groupoid, then $G$ acts on $M=G_{0}$ via the
unit morphism $w=p_{0}:M=G_{0}\rightarrow G_{0}$. Indeed, since $p$
is the unit morphism of Lie group-groupoids, $p$ and $p_{0}$ are Lie
group homomorphisms. Hence $w$ is a Lie group homomorphism. The
composition of the target map of Lie group-groupoid with the
projection $pr_{1}$ gives action $\beta \circ pr_{1}=\phi
:G~_{\alpha }\times _{w}G_{0}\rightarrow G_{0}$ by
$(a,x)=~^{a}x=\beta (a)$. Since $pr_{1}$ and $\beta $ are smooth,
the composition $\beta \circ pr_{1}=\phi $ is also smooth. Now let
us show that the conditions of the action are satisfied. We have
$w(^{a}x)=w(\beta (a))=\beta (a)$, because $w$ is the unit morphism.
That is, the first condition is satisfied. The second condition is
satisfied, namely $^{b}(^{a}x)=~^{b}(\beta (a))=\beta (b)=\beta
(b\circ a)=~^{(b\circ a)}x$. Finally, $^{1_{w(x)}}x=\beta
(1_{w(x)})=x$. Also the interchange law is satisfied. That is,
$(^{b}y)+(^{a}x)=\beta (b)+\beta (a)=\beta (b+a)$ and
$~^{b+a}(y+x)=\beta (b+a)$. Hence it follows
$(^{b}y)+(^{a}x)=~^{b+a}(y+x)$. Thus, the smooth action conditions
are satisfied.
\end{example}

Let $G$ be a Lie group-groupoid. Therefore, we obtain category
$LGGdOp(G)$ of smooth actions of $G$ on Lie groups. A morphism from
$(M,w)$ to $(M^{^{\prime }},w^{^{\prime }})$ is a homomorphism
$f:M\rightarrow M^{^{\prime }} $ of Lie groups satisfying the
conditions $w^{^{\prime }}\circ f=w$ and $f(~^{a}x)=~^{a}f(x)$.

\begin{example}\label{3.2.3}
Let $p:\widetilde{G}\rightarrow G$ be covering morphism of Lie
group-groupoids. Then there is an action of Lie group-groupoid $G$
on Lie group $M=\widetilde{G}_{0}$ via Lie group homomorphism
$w=p_{0}:\widetilde{G}_{0}\rightarrow G_{0}$. Indeed, since $p$ is
the covering morphism of Lie group-groupoids, $p$ and $p_{0}=w$ are
Lie group homomorphisms. Also there exists diffeomorphism
$s_{p}:G~_{\alpha }\times _{p_{0}}\widetilde{G}_{0}\rightarrow
\widetilde{G}$. Since $s_{p}$ and $\widetilde{\beta }$ are smooth,
the composition of $\widetilde{\beta }:\widetilde{G}\rightarrow
\widetilde{G}_{0}$ and $s_{p}$ gives a smooth action $\phi
=\widetilde{\beta }\circ s_{p}:G~_{\alpha }\times
_{p_{0}}\widetilde{G}_{0}\rightarrow \widetilde{G}_{0}$,
$(a,\widetilde{x})\mapsto ~^{a}\widetilde{x}=\widetilde{\beta
}(\widetilde{a})$. By \cite{MHG}, there exists a smooth action of
the underlying Lie groupoid of $G$ on the underlying manifold of
$M=\widetilde{G}_{0}$ via smooth submersion
$w=p_{0}:\widetilde{G}_{0}\rightarrow G_{0}$. So the conditions of
the action are satisfied. Therefore, Lie group-groupoid $G$ acts
smootly on Lie group $M=\widetilde{G}_{0}$ via Lie group
homomorphism $w=p_{0}:\widetilde{G}_{0}\rightarrow G_{0}$.
\end{example}

\begin{example}\label{3.2.4}
Let $G$ be a Lie group-groupoid acting on Lie group $M$ via Lie
group homomorphism $w:M\rightarrow G_{0}$. Then we have action Lie
groupoid $G\ltimes M$ whose objects manifold is Lie group $M$.
Further, action Lie groupoid $G\ltimes M$ is a Lie group-groupoid
defined by the operation $(a,x)+(b,y)=(a+b,x+y)$, where the
operation $+$ is defined by the group operation of Lie
group-groupoid $G$. By \cite{MHG}, the covering morphism $p:G\ltimes
M\rightarrow G$ of underlying Lie groupoids of $G\ltimes M$ and $G$
is defined. Now let us show that $p$ is a covering morphism of Lie
group-groupoids. For this, we must show that $p$ preserves the group
structure. Since $p$ is the projection map,
\begin{equation*}
p((a,x)+(b,y))=p(a+b,x+y)=a+b=p(a,x)+p(b,y).
\end{equation*}
Thus, $p$ is a Lie group homomorphism. Consequently, $p$ is a
covering morphism of Lie group-groupoids.
\end{example}

Let us now give second main result of this paper.

\begin{theorem}
Let $G$ be a Lie group-groupoid. Then the category $LGGdCov(G)$ of
the coverings of $G$ and the category $LGGdOp(G)$ of the actions of
$G$ on Lie groups are equivalent.
\end{theorem}

\begin{proof}
Let us define a functor $\Gamma :LGGdOp(G)\rightarrow LGGdCov(G)$ as
follows. Suppose $\phi :G~_{\alpha }\times _{w}M\rightarrow M$,
$(a,x)\mapsto \phi (a,x)=~^{a}x$ be smooth action of Lie
group-groupoid $G$ on a Lie group $M$ via Lie group homomorphism
$w:M\rightarrow G_{0}$. Then, from Example \ref{3.2.4}, the Lie
group-groupoid $G\ltimes M$ whose objects manifold is $M$ is
defined. Since $p:G\ltimes M\rightarrow G$ is defined by
$(a,x)\mapsto a$ on the morphisms and by $w$ on the objects, it is a
covering morphism of Lie group-groupoids. That is, $\Gamma (M,w)$ is
the covering morphism of Lie group-groupoids. If $(M,w)$ and
$(M^{^{\prime }},w^{^{\prime }})$ are smooth actions then $\Gamma
(M,w)$ and $\Gamma (M^{^{\prime }},w^{^{\prime }})$ are smooth
covering morphisms of Lie group-groupoids. Let us denote these
smooth covering morphisms by $p:G\ltimes M\rightarrow G$ and
$q:G\ltimes M^{^{\prime }}\rightarrow G$, respectively. If
$f:M\rightarrow M^{^{\prime }}$ is a morphism of the smooth actions,
then $\Gamma (f)=r$ is also a morphism of smooth covering morphisms
with $r_{0}=f$ and $r=1\times f$. However, if $f:M\rightarrow
M^{^{\prime }}$ and $g:M^{^{\prime }}\rightarrow N$ are morphisms of
the smooth actions, then $\Gamma (g\circ f)=\Gamma (g)\circ \Gamma
(f)$. If we denote by $\Gamma (M,w)=G\ltimes M$, $\Gamma
(M^{^{\prime }},w^{^{\prime }})=G\ltimes M^{^{\prime }}$, $\Gamma
(N,w^{^{\prime \prime }})=G\ltimes N$, $\Gamma (f)=r$ and $\Gamma
(g)=r^{^{\prime }}$, then we have $g\circ f:M\rightarrow N$ and
hence $\Gamma (g\circ f)=r^{^{\prime }}\circ r=\Gamma (g)\circ
\Gamma (f)$. Thus, $\Gamma $ is a functor.

Secondly, let us define a functor $\Phi :LGGdCov(G)\rightarrow
LGGdOp(G)$ as follows. Let $p:\widetilde{G}\rightarrow G$ be a
covering morphism of Lie group-groupoids. Let us take
$M=\widetilde{G}_{0}$ and $w=p_{0}:\widetilde{G}_{0}\rightarrow
G_{0}$. One can obtain a smooth action
$(M=\widetilde{G}_{0},w=p_{0})$ by Example \ref{3.2.3}. That is,
$\Phi (p)$ is a smooth action of the Lie-group-groupoid $G$ on a Lie
group. If $p:\widetilde{G}\rightarrow G$ and $q:H\rightarrow G$ are
covering morphisms of Lie group-groupoids, then $\Phi (p)$ and $\Phi
(q)$ are actions of Lie group-groupoid $G$ on Lie groups
$\widetilde{G}_{0}$ and $H_{0}$  via Lie group homomorphisms $p_{0}$
and $q_{0}$, respectively. Let $(\widetilde{G}_{0},p_{0})$ and
$(H_{0},q_{0})$ be these smooth actions, respectively. If $p $ and
$q$ are covering morphisms of Lie group-groupoids, then
$r:\widetilde{G}\rightarrow H$ is also a covering morphism of Lie
group-groupoids. Thus, if $r$ is the morphism of the covering
morphisms of Lie group-groupoids, then $\Phi (r)=f$ is also morphism
of the smooth actions with $r_{0}=f$. Indeed, the diagram
\[
\xymatrix{\widetilde G_{0} \ar[rr]^{r_{0}=f} \ar[ddr]_{p_{0}} &&
H_{0} \ar[ddl]^{q_{0}}\\\\ &G_{0}}
\]
is commutative, because $r$ is the morphism of covering morphisms of
Lie group-groupoids. Furthermore, since $r$ is the morphism of
covering morphisms of Lie group-groupoids, we have $p=q\circ r$ and
$p_{0}=q_{0}\circ r_{0}$. It is easily seen that the action is
preserved by the following diagram.
\[
\xymatrix{G _{\alpha}\times_{p_{0}}\tilde G_{0} \ar[rr]^{\phi}
\ar[dd]_{1\times r_{0}}&&
\tilde{G_{0}}\ar[dd]^{f=r_{0}}\\\\
{G _{\alpha}\times_{q_{0}} H_{0}}\ar[rr]^{\phi'}&& H_{0}}
\]

However, if a morphism from $p:\widetilde{G}\rightarrow G$ to
$q:H\rightarrow G$ is $r:\widetilde{G}\rightarrow H$ and a morphism
from $q:H\rightarrow G$ to $p^{^{\prime }}:H^{^{\prime }}\rightarrow
G$ is $r^{^{\prime }}:H\rightarrow H^{^{\prime }}$, then $\Phi
(r^{^{\prime }}\circ r)=\Phi (r^{^{\prime }})\circ \Phi (r)$. \ If
we denote as $\Phi (p)=(\widetilde{G}_{0},p_{0})$, $\Phi
(q)=(H_{0},q_{0})$, $\Phi (p^{^{\prime }})=(H_{0}^{^{\prime
}},p_{0}^{^{\prime }})$, $\Phi (r)=f$ and $\Phi (r^{^{\prime
}})=f^{^{\prime }}$, then $r^{^{\prime }}\circ
r:\widetilde{G}\rightarrow H^{^{\prime }}$ is a covering morphism of
Lie group-groupoids and hence $\Phi (r^{^{\prime }}\circ
r)=f^{^{\prime }}\circ f=\Phi (r^{^{\prime }})\circ \Phi (r)$. Thus,
$\Phi $ is a functor.

Let us now show that there exist the natural equivalences $\Phi
\Gamma \simeq 1_{LGGdOp(G)}$ and $\Gamma \Phi \simeq
1_{LGGdCov(G)}$. Given a smooth action $(M,w)$, there exist the Lie
group-groupoid $G\ltimes M$ whose the objects manifold is Lie group
$(G\ltimes M)_{0}=M$, and the covering morphism $p:G\ltimes
M\rightarrow G$ of Lie group-groupoids. Furthermore, $\Phi (\Gamma
(M,w))$ gives a smooth action of Lie group-groupoid $G$ on Lie group
$(G\ltimes M)_{0}=M$  via smooth group homomorphism
$p_{0}=w:(G\ltimes M)_{0}=M\rightarrow G_{0}$. That is, $\Phi
(\Gamma (M,w))=(M,w)$. Thus, we have obtain $\Phi \Gamma
=1_{LGGdOp(G)}$.

Conversely, if $p:\widetilde{G}\rightarrow G$ is a covering morphism
of Lie group-groupoids, then $\Phi (p)$ is a smooth action $\phi
:G~_{\alpha }\times _{p_{0}}\widetilde{G}_{0}\rightarrow
\widetilde{G}_{0}$ of the Lie group-groupoid $G$ on Lie group
$\widetilde{G}_{0}$  via smooth group homomorphism
$p_{0}:\widetilde{G}_{0}\rightarrow G_{0}$. Furthermore, $\Gamma
(\Phi (p))$ is also covering morphism of Lie group-groupoids, where
the manifold of objects is $\widetilde{G}_{0}$ and the manifold of
morphisms is $G\ltimes \widetilde{G}_{0}$. Now let us define natural
transformation $T^{^{\prime }}:1_{LGGdCov(G)}\rightarrow \Gamma \Phi
$. If $p:\widetilde{G}\rightarrow G $ is the covering morphism of
Lie group-groupoids, then the map $T_{p}^{^{\prime
}}:\widetilde{G}\rightarrow \Gamma \Phi (p)=G\ltimes
\widetilde{G}_{0}$ is defined by identity on objects and by
$\widetilde{a}\mapsto (a,\widetilde{x})$ on morphisms, where
$\widetilde{a}$ is the lifting of $a$ and $\widetilde{x}$ is the
source of $a$. Since $p$ and $p^{^{\prime }}$ are smooth covering
morphisms, $T_{p}^{^{\prime }}$ is also a smooth covering morphism
from \cite{MHG}. For $\widetilde{a}\in \widetilde{G}$,
$p(\widetilde{a})=a$ and $p^{^{\prime }}(T_{p}^{^{\prime
}}(a))=p^{^{\prime }}(a,\widetilde{x})=a$. That is, the following
diagram is commutative.
\[
\xymatrix{&& {G\ltimes\tilde G_{0}} \ar[dd]^{p'} \\ \\
\tilde G \ar[uurr]^{T^{'}_{p}} \ar[rr]_{p} && G}
\]

Thus, $T_{p}^{^{\prime }}$ is a morphism of the covering morphisms
of Lie group-groupoids. If $q:H\rightarrow G$ is another covering
morphism of Lie group-groupoids, then the following diagram is
commutative.
\[
\xymatrix{\tilde G \ar[rr]^{T'_{p}} \ar[dd]_{r}&&
G\ltimes \tilde{G_{0}}\ar[dd]^{\Gamma\Phi(r)=1\times r_{0}}\\\\
{H}\ar[rr]^{T'_{q}}&& {G\ltimes H_{0}}}
\]
Obviously, the inverse of $T_{p}^{^{\prime }}$ is the morphism
$(T_{p}^{^{\prime }})^{-1}:G\ltimes \widetilde{G}_{0}\rightarrow
\widetilde{G} $ defined by identity on objects and by
$(a,\widetilde{x})\mapsto \widetilde{a}$ on morphisms. Thus,
$T^{^{\prime }}$ is a natural equivalence. So it follows
$1_{LGGdCov(G)}\simeq \Gamma \Phi $.
\end{proof}

\end{document}